\documentclass{amsart}
\usepackage{amssymb,amsthm, amsmath, amsfonts, amsbsy}
\usepackage{graphics}
\usepackage{hyperref}
\usepackage[all]{xy}
\usepackage{enumerate}
\usepackage[mathscr]{eucal}
\usepackage{bbm}
\usepackage{indentfirst}
\usepackage{xcolor}

\theoremstyle{plain}
\newtheorem{theorem}{Theorem}[section]
\newtheorem{lemma}[theorem]{Lemma}
\newtheorem{proposition}[theorem]{Proposition}

\newtheorem{corollary}[theorem]{Corollary}
\numberwithin{equation}{section}

\theoremstyle{definition}

\newtheorem{example}[theorem]{Example}
\newtheorem{remark}[theorem]{Remark}

\DeclareMathOperator{\Module}{-Mod}
\DeclareMathOperator{\module}{-mod}
\DeclareMathOperator{\fdmod}{-fdmod}
\DeclareMathOperator{\Hom}{Hom}

\DeclareMathOperator{\Ext}{Ext}

\newcommand{\C}{{\mathscr{C}}}

\newcommand{\BS}{{\mathbf{S}}}
\newcommand{\BK}{{\mathbf{K}}}
\newcommand{\BD}{{\mathbf{D}}}
\newcommand{\BQ}{{\mathbf{Q}}}
\newcommand{\BR}{{\mathbf{R}}}
\newcommand{\BB}{{\mathbf{B}}}

\newcommand{\ID}{{\mathbf{Id}}}
\newcommand{\N}{{\mathbb{N}}}
\newcommand{\OI}{{\mathscr{OI}}}
\newcommand{\FI}{{\mathscr{FI}}}
\newcommand{\VI}{{\mathscr{VI}_q}}
\newcommand{\mi}{{\mathfrak{I}}}
\newcommand{\tor}{{\mathrm{tor}}}
\newcommand{\se}{{\mathrm{sec}}}
\newcommand{\loc}{{\mathrm{loc}}}

\title{Adjoint functors on the representation category of $\OI$}

\author{Wee Liang Gan}
\address{Department of Mathematics, University of California, Riverside, CA 92521, USA}
\email{wlgan@ucr.edu}

\author{Liping Li}
\address{LCSM (Ministry of Education), School of Mathematics and Statistics, Hunan Normal University, Changsha, Hunan 410081, China.}
\email{lipingli@hunnu.edu.cn}

\thanks{The second author is supported by NSF of China grant 11771135, the Hunan Provincial Science and Technology Department grant 2019RS1039, and the Hunan Provincial Education Department grant 18A016.}

\begin{document}

\begin{abstract}
In this paper we study adjunction relations between some natural functors on the representation category of the category of finite linearly ordered sets and order-preserving injections.
We also prove that the Nakayama functor induces an equivalence from the Serre quotient of the category of finitely generated modules by the category of finitely generated torsion modules to the category of finite dimensional modules.
\end{abstract}

\maketitle

\section{Introduction}

Adjoint pairs of functors have played an important role in the representation theory of the category $\FI$ of finite sets and injections. Specific examples include the \textbf{shift} functor $\BS$, the \textbf{derivative} functor $\BD$, and the \textbf{negative shift} functor $\BS^{-1}$ introduced in \cite{CEF}, and the \textbf{coinduction} functor $\BQ$ introduced in \cite{GL1}. It was proved in \cite{Gan} that $\BS^{-1}$ is simultaneously the left adjoint of $\BS$ and the right adjoint of $\BD$. It was proved in \cite{GL1} that $\BQ$ is the right adjoint of $\BS$.
Analogously, these functors have been constructed for other concrete combinatorial categories appearing in representation stability theory, such as the category $\OI$ of finite totally ordered sets and order-preserving injections (see \cite{GS} and \cite{GL4}), the category $\VI$ of finite dimensional vector spaces over a finite field $\mathbb{F}_q$ and linear injections (see \cite{Nag1} and \cite{Nag2}). In \cite{GLX}, The authors of this paper and Xi introduced the \textbf{Nakayama} functor and its right adjoint the \textbf{inverse Nakayama} functor for representations of $\FI$ and $\VI$. Moreover, representation theory of abstract combinatorial categories equipped with functors sharing similar properties as these concrete functors have been studied; see \cite{GL2} and \cite{GL3}.

The main goal of this paper is to give a systematical construction of these adjoint functors for the category $\OI$, describe their adjunction relations, and pursue possible applications in representation theory of $\OI$. We note that the combinatorial structure of $\OI$ gives us two functorial operations on the morphisms set. In more details, one can extend a morphism (which is a function) $\alpha: S \to T$ by adding a minimal (or maximal) element $\ast$ into $S$ and $T$ such that the extended morphism preserves $\ast$, and hence we get a faithful \textbf{self-embedding} functor $\boldsymbol{\iota}_{\ast}: \OI \to \OI_+$ such that $\boldsymbol{\iota}_{\ast}$ is bijective restricted on the object sets, where $\OI_+$ is the full subcategory of $\OI$ consisting of nonempty sets. The functor $\boldsymbol{\iota}_{\ast}$ induces a functor from the module category of $\OI_+$ to the module category of $\OI$ via pull-back. This pull-back functor, along with its left and right adjoints, give us a triple of adjoint functors. Dually, one can restrict the above morphism $\alpha$ by removing the minimal (or maximal) elements in $S$ and $T$ (provided that they are nonempty) to obtain a new morphism with smaller domain and codomain, and hence we get a full \textbf{self-submerging} functor $\boldsymbol{\pi}_{\ast}: \OI_+ \to \OI$ such that it is again bijective restricted on the object sets. Similarly, the pull-back of $\boldsymbol{\pi}$ (a functor from the module category of $\OI$ to the module category of $\OI_+$) and its left and right adjoint functors give another triple. Furthermore, we show that the first triple of adjoint functors can be extended to a quintuple of adjoint functors. In a summary, we obtain the following table of functors as well as a theorem describing their adjunction relations. For precise definitions and details, please refer to Sections 2 and 3.

\begin{tabular}{c|ccc}
  Functor & Definition & Description & Exactness\\
  \hline
  $\BS_{\ast}$ & restriction along $\boldsymbol{\iota}_{\ast}$ & $(\BS_{\ast} V)_n = V_{n+1}$ & exact \\
 $\BK_{\ast}$ & $\bigoplus_{n \in \N_+} \Hom_{k\C} (k\C e_n/\mi_{\ast}e_n, -)$ & $(\BK_{\ast} V)_n = (\mathrm{Ann}_{\mi_{\ast}}(V))_n$ & left exact\\
  $\boldsymbol{\Gamma}_{\ast}$ & $k\C_+ \otimes_{k\C} -$ & $(\boldsymbol{\Gamma}_{\ast} V)_n = V_0 \oplus \ldots \oplus V_{n-1}$ & exact\\
  $\BQ_{\ast}$ & $\bigoplus_{n \in \N_+}\Hom_{k\C} (k\C_+ e_n, -)$ & $(\BQ_{\ast} V)_n = V_n \oplus V_{n-1}$ & exact\\
  $\boldsymbol{\Psi}_{\ast}$ & left adjoint of $\boldsymbol{\Gamma}_{\ast}$ & $(\boldsymbol{\Psi}_{\ast} V)_n = \varinjlim ( (k\C /\mathfrak{A}) \otimes_{k\C}  \BD_{\ast} \BS_{\ast}^n V)$ & right exact\\
  $\BR_{\ast}$ & right adjoint of $\BQ_{\ast}$ & $(\BR_{\ast}V)_n = V_n \oplus (\BK_{\ast}V)_{n+1}$ & left exact\\
  $\BB_{\ast}$ & restriction along $\boldsymbol{\pi}_{\ast}$ & $(\BB_{\ast}V)_n = V_{n-1}$ & exact\\
  $\BD_{\ast}$ & $\BS_{\ast} \circ (k\C/\mi_{\ast} \otimes_{k\C} -)$ & $(\BD_{\ast}V)_n = (V/\mi_{\ast}V)_{n+1}$ & right exact
\end{tabular}

\begin{theorem} \label{main result}
Let $\C$ be the full skeletal subcategory of $\OI$ consisting of objects $[n]$, $n \in \N$. Then:
\begin{enumerate}
\item $(\boldsymbol{\Psi}_{\ast}, \, \boldsymbol{\Gamma}_{\ast}, \, \BS_{\ast}, \, \BQ_{\ast}, \, \BR_{\ast})$ is a quintuple of adjoint functors in $\C \Module$;
\item $(\BD_{\ast}, \, \BB_{\ast}, \, \BS_{\ast} \BK_{\ast})$ is a triple of adjoint functors in $\C \Module$;
\item $\BK_{\ast} \BQ_{\ast} \cong \BB_{\ast}$, and $\BD_{\ast} \boldsymbol{\Gamma} \cong \BS_{\ast} \BB_{\ast} \cong \ID_{\C \Module}$.
\end{enumerate}
\end{theorem}

We shall point out that some functors have appeared in literature \cite{GL4, GS}: the shift functor $\BS_{\ast}$, its right adjoint $\BQ_{\ast}$, and the functors $\BK_{\ast}$ and $\BD_{\ast}$ obtained by taking the kernel and cokernel of the natural transformation $\ID_{\C \Module} \to \BS_{\ast}$. However, the interpolations of these functors in terms of the category algebra $k\C$ and its ideals, and many relations among these functors listed in the above theorem, are new as far as we know.

The representation category $\OI \module$ of finitely generated $\OI$-modules is abelian when $k$ is Noetherian (see for instances \cite[Theorem 1.5]{GL2} or \cite[Theorem 1.1.3]{SS}), and it has a natural torsion pair: the full subcategory $\OI \module^{\tor}$ of torsion modules and the full subcategory of torsion free modules. Since the former is an abelian subcategory, one can consider the Serre quotient category $\OI \module / \OI \module^{\tor}$. In \cite{GLX}, the authors showed that the inverse functor $\boldsymbol{\nu}$ and the quasi-inverse Nakayama functor $\boldsymbol{\nu}^{-1}$ form an adjoint pair, and induce the following equivalences
\begin{align*}
\FI \module / \FI \module^{\tor} & \cong \FI \module^{\tor},\\
\VI \module / \VI \module^{\tor} & \cong \VI \module^{\tor},
\end{align*}
when $k$ is a field of characteristic 0. Since the required conditions for the pair $(\boldsymbol{\nu}, \, \boldsymbol{\nu}^{-1})$ described in \cite{GLX} still hold, one can also define these functors in $\OI \module$, and prove the following result:

\begin{theorem}
Let $k$ be a field. Then the adjoint pair $(\boldsymbol{\nu}, \, \boldsymbol{\nu}^{-1})$ induces an equivalence
\[
\OI \module / \OI \module^{\tor} \cong \OI \fdmod,
\]
where $\OI \fdmod$ is the category of finite dimensional $\OI$-modules.
\end{theorem}

We remark that for $\FI$ and $\VI$, the category of finitely generated torsion modules coincides with the category of finite dimensional modules. But for $\OI$, they are different.

This theorem has quite a few interesting consequences. For instances, we deduce that objects in the Serre quotient category are of finite length and of finite injective dimension. Moreover, since the \textbf{section} functor $\OI \module /\OI \module^{\tor} \to \OI \module$ and the \textbf{localization} functor $\OI \module \to \OI \module /\OI \module^{\tor}$ induce an equivalence between the category of \textbf{saturated} $\OI$-modules (equipped with an abelian structure not inherited from that of $\OI \module$) and $\OI \module / \OI \module^{\tor}$, one deduces that finitely generated saturated modules are of finite length and of finite injective dimension, and can give an explicit description of simple saturated modules; see Theorem \ref{simple saturated}.

The paper is organized as follows. In Section 2 we include some preliminary results such as definitions, terminologies, and elementary facts used throughout this paper. In Section 3 we explicitly construct functors listed in the above table, and prove the first theorem. Nakayama functors and their applications are described in the last section.

\section{Preliminaries}

From now on we let $\C$ be the skeletal category of $\OI$, whose objects are $[n]$ with the natural ordering, $n \in \mathbb{N}$, and whose morphisms are strictly increasing functions. By convention, we let $[0] = \emptyset$ and assume that there is a unique morphism $[0] \to [n]$ for any $n \in \N$. Denote by $\C_+$ the full subcategory consisting of non-empty sets. For every morphism $\alpha: [m] \to [n]$, we define its \textbf{degree} $\deg(\alpha)$ to be $n-m$. It is easy to check that $\C$ is a \textbf{graded} category in the sense that every non-invertible morphism can be written as a composite of finitely many morphisms of degree 1 (called \textbf{irreducible} morphisms). Furthermore, the above mentioned decompositions satisfy a quadratic relation. For details, please refer to \cite[Section 2]{GS}.

Let $k$ be a commutative ring. A \textbf{representation} of $\C$ (or a \textbf{$\C$-module}) is a covariant functor from $\C$ to $k \Module$, the category of $k$-modules. Morphisms between $\C$-modules are natural transformations. The category $\C \Module$ is abelian, and has enough projective objects. In particular, free $\C$-modules $k\C([n], -)$, the linearizations of representable functors, are projective, and we denote them by $M(n)$ for $n \in \N$. Given a $\C$-module $V$, we denote its value on the object $[n]$ by $V_n$.

\subsection{Right adjoint functors} \label{subsection:right adjoint}

We recall here a standard construction of right adjoint functors; see, for example, \cite[Definition 1 and Proposition 2]{Gan}.

Suppose that $F:\C\Module \to \C\Module$ is a right exact functor which transforms direct sums to direct sums. Then it has a right adjoint functor $F^\dag$ where for each $\C$-module $V$ and for each $n\in \N$, we define
\[ (F^\dag(V))_n = \Hom_{\C\Module} ( F(M(n)), V ).  \]
For any morphism $\alpha: [m]\to [n]$ in $\C$, we define the induced map $(F^\dag(V))_m \to (F^\dag(V))_n$ to be the map
\[ \Hom_{\C\Module} ( F(M(m)), V ) \to \Hom_{\C\Module} ( F(M(n)), V ),
\qquad \phi \mapsto \phi\circ F(\rho_\alpha),\]
where $\rho_\alpha$ is the $\C$-module homomorphism $M(n) \to M(m)$ defined by sending each morphism $\beta \in M(n)_r$ (for any $r\in\N$) to $\beta\circ \alpha \in M(m)_r$.

\subsection{Self-embedding functors}

Let $\alpha: [m] \to [n]$ be a morphism in $\C$. Then one can define a new morphism $\tilde{\alpha}: [m+1] \to [n+1]$ such that
\[
\tilde{\alpha}(i) = \begin{cases}
1, & i=1;\\
\alpha(i-1) + 1, & i \in [m+1]\setminus \{1\}
\end{cases}
\]
In particular, the unique morphism $[0] \to [1]$ is sent to the canonical inclusion $[1] \to [2]$. It is easy to check that the above construction defines a functor $\boldsymbol{\iota}_a: \C \to \C_+$ such that $\boldsymbol{\iota}_a ([n]) = [n+1]$ and $\boldsymbol{\iota}_a(\alpha) = \tilde{\alpha}$. Furthermore, the functor $\boldsymbol{\iota}_a$ is faithful, and there is a natural transformation $\boldsymbol{\rho}^a: \ID_{\C} \to \boldsymbol{\iota}_a$ where $\boldsymbol{\rho}^a_n: [n] \to [n+1]$ is the map sending $i \in [n]$ to $i+1 \in [n+1]$ (we let $\boldsymbol{\rho}^a_0$ be the unique map $[0] \to [1]$) and making the following diagram commute:
\[
\xymatrix{
[m] \ar[r]^{\alpha} \ar[d]^{\boldsymbol{\rho}_m^a} & [n] \ar[d]^{\boldsymbol{\rho}_n^a}\\
[m+1] \ar[r]^{\tilde{\alpha}} & [n+1].
}
\]

Intuitively, $\boldsymbol{\iota}_a$ comes from the following natural construction: we add another element $-\infty$ to each object in $\C$ to obtain a new category $\tilde{\C}$ isomorphic to $\C_+$, and convert every morphism $\alpha: [m] \to [n]$ in $\C$ to a morphism $\tilde{\alpha}: [m] \cup \{-\infty\} \to [n] \cup \{-\infty\}$ in $\tilde{\C}$ such that $\tilde{\alpha}(-\infty) = -\infty$ and $\tilde{\alpha}(i) = \alpha(i)$ for $i \in [m]$. Now we identify $\tilde{\C}$ with $\C_+$ to obtain the definitions of $\boldsymbol{\iota}_a$ and $\boldsymbol{\rho}^a$.

There is a dual construction by adding a maximal element $\infty$ to each object in $\C$. In this way we obtain another self-embedding functor $\boldsymbol{\iota}_b: \C \to \C_+$ such that $\boldsymbol{\iota}_b ([n]) = [n+1]$ and $\boldsymbol{\iota}_b(\alpha) = \bar{\alpha}$ with $\bar{\alpha}(i) = \alpha(i)$ for $i \in [m]$. In particular, the unique morphism $[0] \to [1]$ is mapped to the function $+1: [1] \to [2]$. The functor $\boldsymbol{\iota}_b$ is also faithful, and its corresponded natural transformation $\boldsymbol{\rho}^b: \ID_{\C} \to \boldsymbol{\iota}_a$ is a sequence of maps $\boldsymbol{\rho}^b_n: [n] \to [n+1]$ sending $i \in [n]$ to $i \in [n+1]$ (we let $\boldsymbol{\rho}^b_0$ be the unique map $[0] \to [1]$).

The self-embedding functors $\boldsymbol{\iota}_a$ and $\boldsymbol{\iota}_b$ induce two shift functors in the representation category. We will introduce the category algebra, and use the classical tensor and hom construction to define their adjoint functors.

\subsection{Self-submerging functors}

Now we construct a functor $\boldsymbol{\pi}_a: \C_+ \to \C$ as follows. For objects, one has $\boldsymbol{\pi}_a([n]) = [n-1]$; for a morphism $\alpha: [m] \to [n]$, one lets $\boldsymbol{\pi}_a(\alpha)$ be the composite of the following maps:
\[
\xymatrix{
[m-1] \ar[r]^-{\sim} & \{2, \, 3, \, \ldots, \, m\} \ar[r]^-{\alpha_{\mathrm{res}}} & \{2, \, 3, \, \ldots, \, n \} \ar[r]^-{\sim} & [n-1]
}
\]
where $\alpha_{\mathrm{res}}$ is the restriction of $\alpha$ on the subset, and the two isomorphisms are obtained by adding 1 and subtracting 1 respectively. This is well defined since the image of $\{2, \, \ldots, \, m\}$ under $\alpha$ is contained in $\{2, \, \ldots, \, n\}$. Intuitively, we obtain $\boldsymbol{\pi}_a$ by removing the minimal element 1 from each object in $\C_+$. The sequence $\boldsymbol{\rho}^a_n$, $n \geqslant 1$, forms a natural transformation $\boldsymbol{\pi}_a \to \ID_{\C_+}$.

Dually, one can remove the maximal element from each object in $\C_+$ to get another self-submerging functor $\boldsymbol{\pi}_b: \C_+ \to \C$. Explicitly, $\boldsymbol{\pi}_b([n]) = [n-1]$, and for $\alpha: [m] \to [n]$, one lets $\boldsymbol{\pi}_b(\alpha)$ be the restriction of $\alpha$. In this case, the sequence of canonical inclusions $\boldsymbol{\rho}_n^b$, $n \geqslant 1$, forms a natural transformation $\boldsymbol{\pi}_b \to \ID_{\C_+}$.

The self-submerging functors are full, and induce two functors in $\C \Module$, called \textbf{negative shift} functors. Accordingly, their left and right adjoint functors exist, and will be studied in next section.

The following lemma collects some elementary facts about these functors.

\begin{lemma} \label{embedding and submerging}
Notation as above. Then:
\begin{enumerate}
\item $\boldsymbol{\pi}_a \boldsymbol{\iota}_a \cong \ID_{\C} \cong \boldsymbol{\pi}_b \boldsymbol{\iota}_b$;
\item $\boldsymbol{\iota}_a \boldsymbol{\iota}_b (\alpha) = \boldsymbol{\iota}_b \boldsymbol{\iota}_a (\alpha)$ for every morphism $\alpha$ in $\C$;
\item $\boldsymbol{\pi}_a \boldsymbol{\pi}_b (\alpha) = \boldsymbol{\pi}_b \boldsymbol{\pi}_a (\alpha)$ for every morphism $\alpha: [m] \to [n]$ in $\C_+$ with $m \geqslant 2$;
\item for any $\alpha: [m] \to [n]$ in $\C_+$, $(\boldsymbol{\iota}_a \boldsymbol{\pi}_b)^m (\alpha)$ is the canonical inclusion from $[m]$ to $[n]$.
\end{enumerate}
\end{lemma}

\begin{proof}
A direct check by definitions.
\end{proof}

\subsection{Category algebra}

In this subsection we define the category algebra $k\C$ so that we can interpret the adjoint triples in $\C \Module$ induced by the self-embedding functors and self-submerging functors by the classical tensor-hom adjunction. The category algebra $k\C$ as a free $k$-module is spanned by all morphisms in $\C$, and the multiplication is determined by the product of two basis elements via the following rule:
\[
\alpha \cdot \beta =
\begin{cases}
\alpha \circ \beta, & \text{if they can be composed;}\\
0, & \text{else.}
\end{cases}
\]
In this way we get an associative (non-unital) $k$-algebra. A $\C$-module $V$ gives a $k\C$-module $\bigoplus_{n \in \N} V_n$. Conversely, a $k\C$-module $V$ is a $\C$-module if and only if $V = \bigoplus_{n \in \N} e_n V$, where $e_n: [n] \to [n]$ is the identity morphism. In other words, we may identify $\C \Module$ as a full subcategory of $k\C \Module$. Basic knowledge about category algebras is described in \cite{We} and \cite{Xu}.

Let $\mi_a$ be the $k$-submodule of $k\C_+$ spanned by morphisms $\alpha: [m] \to [n]$ such that $\alpha(1) \neq 1$, $m, n \in \N$. Dually, let $\mi_b$ be the $k$-submodule of $k\C_+$ spanned by morphisms $\alpha: [m] \to [n]$ such that $\alpha(m) \neq n$. For $n \geqslant 1$ and $i \in [n+1]$, by $\alpha_{n, i}$ we mean the morphism $[n] \to [n+1]$ such that $i$ is not contained in its image. For example, if $n \geqslant 1$, then $\alpha_{n, n+1} = \boldsymbol{\rho}^b_n$ is the canonical inclusion $[n] \to [n+1]$, and $\alpha_{n, 1} = \boldsymbol{\rho}^a_n$.

\begin{lemma}
Notation as above. Then $\mi_a$ and $\mi_b$ are two-sided ideals of $k\C_+$ generated by $\{ \alpha_{n, 1} \mid n \geqslant 1\}$ and $\{ \alpha_{n, n+1} \mid n \geqslant 1\}$ respectively.
\end{lemma}

\begin{proof}
We prove the conclusion for $\mi_a$. Note that the set of morphisms $\alpha: [n] \to [n+1]$ with $n \geqslant 1$ and $\alpha(1) \neq 1$ form a two-sided ideal of the category $\C_+$, so $\mi_a$ is a two-sided ideal of $k\C_+$. Furthermore, every non-invertible morphism in $\C_+$ can be written as a finite composite of irreducible morphisms, and since $\alpha(1) \neq 1$, at least one irreducible morphism must be of the form $\alpha_{n, 1}$. Thus $\mi_a$ is generated by the set $\{ \alpha_{n, 1} \mid n \geqslant 1\}$.
\end{proof}

\begin{remark} \normalfont
Actually, by \cite[Section 2 and Subsection 3.5]{GS}, the left ideal (or right ideal) of $k\C_+$ generated by $\{ \alpha_{n, 1} \mid n \geqslant 1\}$ is also $\mi_a$. The same conclusion holds for $\mi_b$ as well. More explicitly, for $\ast \in \{a, \, b\}$,
\[
\mi_{\ast} = \bigoplus_{i \geqslant 1} k\C \boldsymbol{\rho}^{\ast}_i = \bigoplus_{i \geqslant 1} \boldsymbol{\rho}_i^{\ast} (k\C).
\]
\end{remark}

\section{The adjoint quintuple and triple}

In this section we construct all functors in the table, and describe adjunction relations among them.

\subsection{The shift functors $\BS_{\ast}$ and derivatives}

The shift functor $\BS_a: \C_+ \Module \to \C \Module$ is defined to be the pull back of $\boldsymbol{\iota}_a: \C \to \C_+$. Explicitly, one has $\BS_a V = V \circ \boldsymbol{\iota}_a$. The natural transformation $\boldsymbol{\rho}^a: \ID_{\C} \to \boldsymbol{\iota}_a$ gives a natural transformation $(\boldsymbol{\rho}^a)^{\ast}: \ID_{\C \Module} \to \BS_a$, and hence we can define the functors $\BK_a$ and $\BD_a$, which are the kernel and cokernel of $(\boldsymbol{\rho}^a)^{\ast}$ respectively. The functor $\BD_a$ is called the derivative functor of $\BS_a$. Similarly, one can define functors $\BS_b$, $\BK_b$, and $\BD_b$, using the self-embedding functor $\boldsymbol{\iota}_b$ and the natural transformation $\boldsymbol{\rho}^b$.
From Lemma \ref{embedding and submerging}, we see that $\BS_a \circ \BS_b \cong \BS_b \circ \BS_a$.

Properties of shift functors and derivative functors have been extensively studied; see \cite{GL2}. In this paper we are mainly interested in illustrating these functors by classical tensor or hom functors. To this end, we observe that $\boldsymbol{\iota}_a$ induces an injective algebra homomorphism $\underline{\boldsymbol{\iota}}_a: k\C \to k\C_+$, and $\BS_a$ is the restriction along $\underline{\boldsymbol{\iota}}_a$. With this simple observation, we have:

\begin{proposition}
Let $V$ be a $\C_+$-module, and let $\ast \in \{a, \, b\}$. Then:
\begin{enumerate}
\item $\BS_{\ast} \mi_{\ast} e_n \cong M(n)$ for $n \geqslant 1$.
\item In the exact sequence $0 \to \BK_{\ast} V \to V \to \BS_{\ast} V \to \BD_{\ast} V \to 0$, one has
\begin{align*}
\BK_{\ast}V & = \mathrm{Ann}_{{\mi}_{\ast}}(V) = \bigoplus_{n \in \N} \{v \in V_n \mid \mi_{\ast} \cdot v = 0 \} \cong \bigoplus_{n \in \N} \Hom_{k\C} (k\C e_n / \mi_{\ast} e_n, V),\\
\BD_{\ast} V & = \BS_{\ast} V/\BS_{\ast} \mi_{\ast} V \cong \BS_{\ast} (V/\mi_{\ast} V) \cong \BS_{\ast} (k\C /\mi_{\ast} \otimes_{k\C} V).
\end{align*}
\end{enumerate}
\end{proposition}

In other words, $\BK_{\ast}$ is the annihilator functor $\mathrm{Ann}_{\mi_{\ast}} (-)$, and $\BD_{\ast}$ is the composite of $k\C/\mi_{\ast} \otimes_{k\C} - $ and $\BS_{\ast}$. We also deduce that the natural map $V \to \BS_{\ast}V$ is injective if and only if $V \cong \BS_{\ast} \mi_{\ast} V$.

\begin{proof}
We prove the conclusions for $\ast = a$.

(1): For $s \geqslant n \geqslant 1$, one has
\[
(\BS_a (\mi_a e_n))_s = (\mi_a e_n)_{s+1} = \langle \alpha: [n] \to [s+1] \mid \alpha(1) \neq 1 \rangle
\]
as free $k$-modules. Since $\alpha(1) \neq 1$, one can identify it with a morphism $\bar{\alpha}: [n] \to [s]$ with $\bar{\alpha}(i) = \alpha(i) - 1$, and the map $\alpha \mapsto \bar{\alpha}$ gives a $k$-module isomorphism between $\BS_{\ast} \mi_{\ast} e_n$ and $M(n)$. It is easy to check it is actually an $\C$-module homomorphism, keeping in mind that for a morphism $\alpha$ in $\C$, its action on a shift module $\BS_a V$ is given by the action of $\boldsymbol{\iota}_a(\alpha)$ on $V$. This proves (1).

(2): Denote the image of the natural map $V \to \BS_a V$ by $V'$. We claim that $V'$ coincides with the submodule $\BS_a \mi_a V$ of $\BS_a V$. Since this natural map is given by a sequence $\boldsymbol{\rho}_n^a$, $n \geqslant 1$, we have
\[
V' = \bigoplus_{n \geqslant 1} \boldsymbol{\rho}_n^a \cdot V_n,
\]
in particular, $V'_n = \boldsymbol{\rho}_n^a \cdot V_n$. On the other hand, $(\BS_a \mi_aV)_n = (\mi_aV)_{n+1} = e_{n+1} \mi_aV$. Thus it is enough to show that $\boldsymbol{\rho}_n^a \cdot V_n = e_{n+1} \mi_aV$. But $e_{n+1} \mi_a$ contains $\boldsymbol{\rho}^a_n$. Consequently, $\boldsymbol{\rho}_n^a \cdot V_n \subseteq e_{n+1} \mi_a V$.

To show the inclusion of the other direction, we note that $\mi_a$ is generated by $\boldsymbol{\rho}_n^a$ as a left ideal (see Remark 2.3). That is, $\mi_a$ is the sum of $k\C \boldsymbol{\rho}_i^a$ for $i \geqslant 1$. Consequently,
\[
e_{n+1} \mi_a V = e_{n+1}(\sum_{i \in \mathbb{N}} (k\C \boldsymbol{\rho}_i^a) \cdot (\bigoplus_{j \in \mathbb{N}} V_j)) = (\sum_{i \in \mathbb{N}} (e_{n+1} k\C \boldsymbol{\rho}_i^a) \cdot (\bigoplus_{j \in \mathbb{N}} V_j)) = \sum_{i \leqslant n} (e_{n+1} k\C\boldsymbol{\rho}_i^a \cdot V_i)
\]
where the last identity comes from the following two observations: $e_{n+1} k\C \boldsymbol{\rho}_i^a = 0$ for $i > n$, and $\boldsymbol{\rho}_i^a \cdot V_j = 0$ whenever $i \neq j$. But $\mi_a$ as a right ideal is also generated by those $\boldsymbol{\rho}_i^a$ (see Remark 2.3). In particular, for $i < n$, any morphism $\alpha$ in $e_{n+1} k\C \boldsymbol{\rho}_i^a$ from $[i]$ to $[n+1]$ can be written as a composite $\boldsymbol{\rho}_n^a \circ \beta \in \boldsymbol{\rho}_n^a k\C$. Therefore, for $i < n$,
\[
(e_{n+1} k\C \boldsymbol{\rho}_i^a) \cdot V_i \subseteq (e_{n+1} \boldsymbol{\rho}_n^a k\C) \cdot V_i = (\boldsymbol{\rho}_n^a e_n k\C e_i) \cdot V_i = \boldsymbol{\rho}_n^a \cdot (k\C([i], [n]) \cdot V_i) \subseteq \boldsymbol{\rho}_n^a \cdot V_n;
\]
and for $i =n$, $e_{n+1} k\C \boldsymbol{\rho}_n^a \cdot V_n = \boldsymbol{\rho}_n^a \cdot V_n$ as well. Therefore, $e_{n+1} \mi_a V = \boldsymbol{\rho}_n^a \cdot V_n$, and hence $V' = \BS_a \mi_a V$ as claimed.

Now we consider $\BK_a V$. For $n \geqslant 1$, $(\BK_aV)_n$ is the kernel of the map induced by $\boldsymbol{\rho}_n^a$. That is,
\[
(\BK_aV)_n = \{ v \in V_n \mid \boldsymbol{\rho}_n^a \cdot v = 0\}.
\]
Clearly, for $i \neq n$, $\boldsymbol{\rho}_i^a \cdot v = 0$. Thus one can rewrite the about identity as
\[
(\BK_aV)_n = \{ v \in V_n \mid \boldsymbol{\rho}_i^a \cdot v = 0, \, i \geqslant 1\}.
\]
But $\mi_a$ is the sum of $k\C \boldsymbol{\rho}_i^a$, so
\[
(\BK_a V)_n = \{ v \in V_n \mid \mi_a \cdot v = 0 \}
\]
and hence $\BK_a V = \mathrm{Ann}_{\mi_a}(V)$. Furthermore, the isomorphism
\[
(\BK_a V)_n \cong \Hom_{k\C} (k\C e_n/\mi_a e_n, V)
\]
comes from the observation that the module homomorphism $f: e_n \mapsto v \in V_n$ in $\Hom_{k\C} (k\C e_n, V)$ gives rise to a unique and well defined module homomorphism in $\Hom_{k\C} (k\C e_n/\mi_a e_n, V)$ if and only if $\mi_a e_n v = 0$, which is equivalent to the condition that $\boldsymbol{\rho}_n^a v = 0$ since $\mi_a e_n = k\C \boldsymbol{\rho}_n^a$.

Finally, one has
\[
\BD_a V = \BS_aV/V' = \BS_aV/\BS_a \mi_a V \cong \BS_a(V/\mi_a V) \cong \BS_a (k\C/\mi_a \otimes_{k\C} V)
\]
where the first isomorphism comes from the exactness of $\BS_a$ and the second one is well known.
\end{proof}

\subsection{The left adjoint of $\BS_{\ast}$}

As before, let $\ast \in  \{a, \, b\}$. Since $\BS_{\ast}: \C_+ \Module \to \C \Module$ is the restriction along the injective algebra homomorphism $\underline{\boldsymbol{\iota}}_{\ast}: k\C \to k\C_+$, it has a left adjoint
\[   \C \Module \to \C_+ \Module, \qquad
V\mapsto  k\C_+ \otimes_{k\C} V, \]
called the induction functor. Note that the right $\C$-module structure of $k\C_+$ is induced by $\underline{\boldsymbol{\iota}}_{\ast}$.

We give here another construction of a left adjoint of $\BS_\ast$. For any $\C$-module $V$, we define
a $\C_+$-module $\boldsymbol{\Gamma}_* V$ by
\[
(\boldsymbol{\Gamma}_\ast V)_n = V_0 \oplus \cdots \oplus V_{n-1} \qquad \mbox{ for each } n\geqslant 1;
 \]
for any morphism $\alpha: [m] \to [n]$ in $\C_+$, the induced map $(\boldsymbol{\Gamma}_* V)_m \to (\boldsymbol{\Gamma}_* V)_n$ is defined as follows. If $\ast = a$, then for each $s\in [m]$, the direct summand $V_{m-s}$ of $(\boldsymbol{\Gamma}_a V)_m$ is mapped to the direct summand
$V_{n - \alpha(s)}$ of $(\boldsymbol{\Gamma}_a V)_n$ via the action of the morphism
\[
[m-s] \to [n-\alpha(s)], \qquad i \mapsto \alpha(s+i)-\alpha(s).
\]
If $\ast = b$, then for each $s \in [m]$,
the direct summand $V_{s-1}$ of $(\boldsymbol{\Gamma}_b V)_m$ is mapped to the direct summand
$V_{\alpha(s)-1}$ of $(\boldsymbol{\Gamma}_b V)_n$ via the action of the morphism
\[
[s-1] \to [\alpha(s)-1], \qquad i \mapsto \alpha(i).
\]
It is easy to check that this indeed defines a $\C_+$-module $\boldsymbol{\Gamma}_\ast V$ and hence a functor $\boldsymbol{\Gamma}_\ast : \C\Module\to \C_+\Module$.
Moreover, for each $m\in \N$, one has an isomorphism of $\C_+$-modules
$\boldsymbol{\Gamma}_\ast M(m) \cong M(m+1)$ defined as follows: if $\ast=a$, then for each $n$, let
\[ f: M(m+1)_n \to (\boldsymbol{\Gamma}_\ast M(m))_n = M(m)_0 \oplus \cdots \oplus M(m)_{n-1}
\]
be the map sending each morphism $\alpha:[m+1]\to [n]$ in $M(m+1)_n$ to the morphism
\[ f(\alpha) : [m] \to [n-\alpha(1)], \qquad i\mapsto  f(i+1)-f(1) \]
in the direct summand $M(m)_{n-\alpha(1)}$ of $(\boldsymbol{\Gamma}_\ast M(m))_n$; similarly if $\ast=b$.

\begin{remark} \normalfont
The functor $\boldsymbol{\Gamma}_{\ast}$ introduced here is different from that described in \cite[Subsection 9.1]{GS}. Actually, it corresponds to the induction functor defined in \cite[Subsection 13.4]{GS}.
\end{remark}

\begin{proposition}
The functor $\boldsymbol{\Gamma}_\ast$ is exact and is left adjoint to the functor $\BS_\ast$.
\end{proposition}
\begin{proof}
We prove the proposition for $\ast=a$; the case $\ast=b$ is similar.

Exactness of the functor $\boldsymbol{\Gamma}_a$ is clear by construction and, moreover, it transforms direct sums to direct sums. Therefore by Section \ref{subsection:right adjoint} it has a right adjoint functor $\boldsymbol{\Gamma}_a^\dag$. Identifying $\boldsymbol{\Gamma}_a M(n)$ with $M(n+1)$ as above, we have, for any $\C_+$-module $V$ and $n$:
\[ (\boldsymbol{\Gamma}_a^\dag V)_n = \Hom_{k\C_+} (M(n+1), V) \cong V_{n+1}, \]
where the isomorphism is $\phi \mapsto \phi(e_{n+1})$. Now consider any morphism $\alpha: [n]\to [\ell]$ in $\C$. Then by Section \ref{subsection:right adjoint} we have $\rho_\alpha: M(\ell) \to M(n)$ and hence $\boldsymbol{\Gamma}_a (\rho_\alpha) :
\boldsymbol{\Gamma}_a M(\ell) \to \boldsymbol{\Gamma}_a M(n)$, which is identified with a morphism $M(\ell+1)\to M(n+1)$. From definitions, this morphism sends $e_{\ell+1}\in M(\ell+1)_{\ell+1}$ to $ \boldsymbol{\iota}_a(\alpha) \in M(n+1)_{\ell+1}$; in more detail, $e_{\ell+1}\in M(\ell+1)_{\ell+1}$ is identified with
$e_\ell \in M(\ell)_\ell$ in $(\boldsymbol{\Gamma}_a M(\ell))_{\ell+1}$,
and $\boldsymbol{\Gamma}_a (\rho_\alpha)$ sends it to
$\alpha\in M(n)_\ell$ in $(\boldsymbol{\Gamma}_a M(n))_{\ell+1}$,
which in turn is identified with $ \boldsymbol{\iota}_a(\alpha) \in M(n+1)_{\ell+1}$.
Therefore if $v=\phi(e_{n+1})$, then $(\phi\circ\boldsymbol{\Gamma}_a (\rho_\alpha))(e_{\ell+1})   =  \boldsymbol{\iota}_a(\alpha)(v)$. Hence $\boldsymbol{\Gamma}_a^\dag V \cong \BS_a V$.
\end{proof}

\subsection{The Left adjoint of $\boldsymbol{\Gamma}_{\ast}$}

To construct the left adjoint of $\boldsymbol{\Gamma}_{\ast}$, we need the following notations. Let $\mathfrak{A}$ be the $k$-submodule of $k\C$ spanned by elements of the form $\alpha - \beta$, where $\alpha$ and $\beta$ are morphisms with the same source and target. Clearly, $\mathfrak{A}$ is a two-sided ideal of $k\C$, called the \textbf{augmentation ideal}.

For any $\C$-module $V$, $\mathfrak{A}V$ is the span of $(\alpha - \beta) \cdot v$ for every $v\in V_m$ where $m<n$ and $\alpha, \beta : [m]\to [n]$.
Consequently, for each $m\ < n$, every morphism $\alpha: [m] \to [n]$ acts on the quotient module $V/\mathfrak{A}V$ in the same way. We define the $k$-module $\varinjlim V$ to be the direct limit of
\[
\xymatrix{
V_0 \ar[r]^-{\boldsymbol{\rho}_0^b} & V_1 \ar[r]^-{\boldsymbol{\rho}_1^b} & V_2 \ar[r]^-{\boldsymbol{\rho}_2^b} & \cdots.
}
\]
Concretely, each element of $\varinjlim V$ is an equivalence class of a sequence $\{v_m, \, v_{m+1}, \, v_{m+2}, \,\ldots \}$ (for some $m$) where each $v_n \in V_n$ and $\boldsymbol{\rho}^b_n \cdot v_n = v_{n+1}$, and where two such sequences are equivalent if they eventually coincide.


\begin{example} \label{example:limit} \normalfont
For each $r\geqslant n$, we have a linear bijection $(M(n) / \mathfrak{A}M(n))_r \xrightarrow{\simeq} k$ induced by sending each morphism $\alpha: [n]\to [r]$ in $M(n)_r$ to $1\in k$. Under this identification, the map
\[ \boldsymbol{\rho}_r^b : (M(n) / \mathfrak{A}M(n))_r \to (M(n) / \mathfrak{A}M(n))_{r+1}\]
is the identity map $k\to k$. It follows that we have an isomorphism $\varinjlim M(n) / \mathfrak{A}M(n) \cong k$.

\end{example}

We also observe that every $\C$-module $V$ gives rise to a functor $\mathcal{S}_b^{V}: \C \to \C \Module$ defined in the following way. For each ordered set $[n]$, let $\mathcal{S}_b^{V}([n]) = \BS_b^n V$. For each morphism $\alpha: [m] \to [n]$ in $\C$, to define the induced $\C$-module homomorphism $\BS_b^m V \to \BS_b^n V$, we need to define for each $r\in\N$ a map
$(\BS_b^m V)_r \to (\BS_b^n V)_r$. Since $(\BS_b^m V)_r = V_{r+m}$ and $(\BS_b^n V)_r = V_{r+n}$, we need to define a map $V_{r+m} \to V_{r+n}$; this map is defined to be the map induced by
\[ [r+m] \to [r+n], \quad i \mapsto \left\{
\begin{array}{ll}
i & \mbox{ if } 1\leqslant i \leqslant r,\\
\alpha(i-r)+r & \mbox{ if } r+1\leqslant i \leqslant r+m.
\end{array}\right. \]
 It is easy to see that $\mathcal{S}_b^{V}$ is indeed a functor and the assignment $V \to \mathcal{S}_b^{V}$ is functorial. In other words, $\mathcal{S}_b$ is a bifunctor from $\C \Module \times \C$ to $\C \Module$. Similarly, one can define $\mathcal{S}_a^{V}: \C \to \C \Module$.

Composed with the derivative functor $\BD_b: \C \Module \to \C \Module$, we obtain a functor $\BD_b \mathcal{S}_b^{V}: \C \to \C \Module$. For $n \in \N$, $(\BD_b \mathcal{S}_b^{V}) ([n]) = \BD_b \BS_b^n V$. Observe that:
\[
(\BD_b \BS_b^n V)_r = \mathrm{coker}( V_{r+n} \longrightarrow V_{r+1+n} ),
\]
where the map is induced by the order-preserving map $[r+n]\to [r+1+n]$ whose image does not contain $r+1\in [r+1+n]$.


Consider the bifunctor defined by the following composition:

\[ \xymatrix{
 \C \Module \times \C \ar[r]^-{\mathcal{S}_b} & \C \Module \ar[r]^-{\BD_b} & \C \Module \ar[rr]^-{k\C/\mathfrak{A} \otimes_{k\C} -} & & \C \Module \ar[r]^-{\varinjlim} & k \Module.} \]
For each $\C$-module $V$, we define the $\C$-module $\boldsymbol{\Psi}_a V$ to be the composition:
\[
\xymatrix{
\C \ar[r]^-{\mathcal{S}_b^V} & \C \Module \ar[r]^-{\BD_b} & \C \Module \ar[rr]^-{k\C/\mathfrak{A} \otimes_{k\C} -} & & \C \Module \ar[r]^-{\varinjlim} & k \Module.
}
\]
Therefore we obtain a functor
\[ \boldsymbol{\Psi}_a  :\C \Module \to \C \Module, \qquad V \mapsto \boldsymbol{\Psi}_a V. \]
Since colimit is computed pointwise, for any $V \in \C \Module$, we have
\[
(\boldsymbol{\Psi}_a V)_n = \varinjlim ( (k\C /\mathfrak{A}) \otimes_{k\C} \BD_b \mathcal{S}_b^V) ([n]) = \varinjlim ( (k\C /\mathfrak{A}) \otimes_{k\C}  \BD_b \BS_b^n V).
\]
Here we write $\boldsymbol{\Psi}_a$ instead of $\boldsymbol{\Psi}_b$ because we will show that $\boldsymbol{\Psi}_a$ defined in this way is the left adjoint of $\boldsymbol{\Gamma}_a$.

\begin{remark} \normalfont
The functor $\boldsymbol{\Psi}_{\ast}$ defined above is not the same as the one described in \cite[Subsection 9.1]{GS}.
\end{remark}

\begin{proposition}
The functor $\boldsymbol{\Psi}_a$ is the left adjoint of $\boldsymbol{\Gamma}_a$.
\end{proposition}

\begin{proof}
For each $n\in \N$, the functor from $\C\Module$ to $k\Module$ sending $V$ to $(\boldsymbol{\Psi}_a V)_n$ is right exact; hence $\boldsymbol{\Psi}_a$ is a right exact functor. By Section \ref{subsection:right adjoint}, the functor $\boldsymbol{\Psi}_a$ has a right adjoint functor $\boldsymbol{\Psi}_a^\dag$ defined by
\[ (\boldsymbol{\Psi}_a^\dag V)_m = \Hom_{k\C} ( \boldsymbol{\Psi}_a (M(m)), V ) \quad \mbox{ for each } m\in \N.  \]
We shall prove that $\boldsymbol{\Gamma}_a$ is isomorphic to  $\boldsymbol{\Psi}_a^\dag$.

First, for each $m,n,r\in\N$, we have:
\[
(\BD_b \BS_b^n M(m))_r = \mathrm{coker}(M(m)_{r+n} \longrightarrow M(m)_{r+1+n} ),
\]
where the map $M(m)_{r+n} \longrightarrow M(m)_{r+1+n}$ is induced by $\alpha_{r+n, r+1}: [r+n]\to [r+1+n]$ (the order-preserving map which does not contain $r+1$ in its range). An order-preserving map $[m]\to [r+1+n]$ factors through $\alpha_{r+n, r+1}$ if and only if $r+1$ is not in its range. Suppose that $\alpha : [m] \to [r+1+n]$ is order-preserving and $\alpha(s)=r+1$ for some $s\in [m]$. Then we have order-preserving maps
\[ [s-1] \to [r], \quad i \mapsto \alpha(i) \]
and
\[ [m-s] \to [n], \quad i \mapsto \alpha(s+i)-\alpha(s). \]
Moreover every such pair of order-preserving maps $[s-1]\to r$ and $[m-s]\to [n]$ arises from a unique order-preserving $\alpha: [m]\to [r+1+n]$ such that $\alpha(s)=r+1$.
We see that there is a natural isomorphism
\[ (\BD_b \BS_b^n M(m))_r \cong \bigoplus_{s=1}^m M(s-1)_r \otimes_k M(m-s)_n. \]
Hence, by Example \ref{example:limit},
\[ \varinjlim ( (k\C /\mathfrak{A} )\otimes_{k\C} \BD_b \BS_b^n M(m) ) \cong
  \bigoplus_{s=1}^m  M(m-s)_n.  \]
Therefore we have a natural isomorphism
\[ \boldsymbol{\Psi}_a (M(m)) \cong \bigoplus_{s=1}^m  M(m-s). \]

Now, for each $m\in\N$ and $s\in [m]$, there is a linear bijection
\[ \Hom_{k\C} (M(m-s), V) \xrightarrow{\simeq} V_{m-s}, \qquad \phi \mapsto \phi(e_{m-s}).\] Hence, for each $m\in \N$, we obtain a linear bijection
\[ (\boldsymbol{\Psi}_a^\dag V)_m \xrightarrow{\simeq} \bigoplus_{s=1}^m V_{m-s} = (\boldsymbol{\Gamma}_a V)_m. \]
It remains to see that these linear bijections for all $m\in \N$ are compatible with the $\C$-module structures of $\boldsymbol{\Psi}_a^\dag V$ and $\boldsymbol{\Gamma}_a V$ so that we have
a natural $\C$-module isomorphism $\boldsymbol{\Psi}_a^\dag V \to \Gamma_a V$. To this end, consider any morphism $\alpha: [m]\to [m']$ in $\C$ and let $\rho_\alpha : M(m') \to M(m)$ be the morphism $\beta \mapsto \beta\circ \alpha$. Then we have the homomorphism
\[ \boldsymbol{\Psi}_a (\rho_a) : \boldsymbol{\Psi}_a(M(m')) \to \boldsymbol{\Psi}_a(M(m)). \]
Recall from above that we have the natural isomorphisms
\[ \boldsymbol{\Psi}_a(M(m'))  \cong \bigoplus_{s'=1}^{m'} M(m'-s') ,
\qquad \boldsymbol{\Psi}_a(M(m)) \cong  \bigoplus_{s=1}^{m} M(m-s).\]
Hence $\boldsymbol{\Psi}_a (\rho_a)$ gives a homomorphism
\[    \bigoplus_{s'=1}^{m'} M(m'-s') \to \bigoplus_{s=1}^{m} M(m-s); \]
let us denote this homomorphism also by $\boldsymbol{\Psi}_a (\rho_a)$.

Let $s'\in [m']$ and consider the identity morphism $e_{m'-s'} \in M(m'-s')_{m'-s'}$.
Let $\beta_0: [s'-1]\to [r]$ be any morphism in $\C$ and consider the morphism $\beta: [m']\to [r+1+(m'-s')]$ defined by
\[ \beta(i) =
\left\{ \begin{array}{ll}
\beta_0(i) & \mbox{ if } 1\leqslant i\leqslant s'-1,\\
r+1+(i-s') & \mbox{ if } s' \leqslant i \leqslant m'.
\end{array} \right. \]
Then $\beta\in M(m')_{r+1+(m'-s')}$ and is a representative of an element $\bar{\beta}$ in
$(\BD_b \BS_b^{m'-s'} M(m'))_r$; this element $\bar{\beta}$ is a representative of $e_{m'-s'}$ under the isomorphism
\[ \varinjlim ( (k\C /\mathfrak{A} )\otimes_{k\C} \BD_b \BS_b^{m'-s'} M(m') ) \cong
  \bigoplus_{t=1}^{m'} M(m'-t)_{m'-s'}.  \]
Now consider $\rho_\alpha(\beta) = \beta\circ\alpha\in M(m)_{r+1+(m'-s')}$.
If $s'$ is not in the range of $\alpha$, then $r+1$ is not in the range of $\rho_\alpha(\beta)$;
in this case, $\boldsymbol{\Psi}_a (\rho_a)$ sends $e_{m'-s'}$ to $0$. On the other hand, if $s'=\alpha(s)$ where $s\in [m]$, then $\rho_\alpha(\beta)(s) = r+1$; in this case,
$\boldsymbol{\Psi}_a (\rho_a)$ sends $e_{m'-s'}$ to the morphism $\alpha_s \in M(m-s)_{m'-s'}$ defined by
\[ \alpha_s : [m-s] \to [m'-s'], \quad i \mapsto \alpha(s+i) - \alpha(s). \]
It follows that the restriction of $\boldsymbol{\Psi}_a (\rho_a)$ to $M(m'-s')$ is the zero map
if $s'$ is not in the range of $\alpha$, and for each $s\in [m]$,
the restriction of $\boldsymbol{\Psi}_a (\rho_a)$ to $M(m'-\alpha(s))$ is the map
\[ M(m'-\alpha(s)) \to M(m-s) , \quad \phi\mapsto \phi\circ \alpha_s. \]
Consider any $v\in V_{m-s}$ and let $\psi :M(m-s)\to V$ be the homomorphism such that $\psi(e_{m-s})=v$. Then
\[ ( \psi \circ \boldsymbol{\Psi}_a (\rho_a) )(e_{m'-\alpha(s)}) = \psi(\alpha_s) = \alpha_s(v) \in V_{m'-\alpha(s)}. \]
But the map $V_{m-s} \to V_{m'-\alpha(s)}$ sending $v$ to $\alpha_s(v)$ is exactly the restriction of the action of $\alpha$ on the direct summand $V_{m-s}$ of the $\C$-module $\boldsymbol{\Gamma}_a V$.
\end{proof}

\subsection{The negative shift functor $\BB_{\ast}$}

For any morphism $\alpha:[n]\to [\ell]$ where $n\geqslant 1$, we define the morphisms
\[ \alpha_a : [n-1]\to [\ell-1]\qquad\mbox{ and }\qquad \alpha_b:[n-1]\to [\ell-1] \]
by
\[ \alpha_a(i) = \alpha(i+1)-1, \qquad \alpha_b(i)=\alpha(i), \qquad \mbox{ for each } i\in [n-1]. \]

Consider any $\C$-module $V$.
For $\ast=a\mbox{ or }b$, we shall define a $\C_+$-module $\BB_\ast V$ as follows.
For each $n\geqslant 1$, let $(\BB_\ast V)_n = V_{n-1}$.
For each morphism $\alpha: [n]\to [\ell]$, we define the induced map
$(\BB_a V)_n \to (\BB_a V)_\ell$ to be the map $V_{n-1}\to V_{\ell-1}$
induced by $\alpha_a$ if $\alpha(1)=1$, otherwise it is the zero map;
similarly, we define the induced map
$(\BB_b V)_n \to (\BB_b V)_\ell$ to be the map $V_{n-1}\to V_{\ell-1}$
induced by $\alpha_b$ if $\alpha(n)=\ell$, otherwise it is the zero map.
This construction is functorial in $V$ and hence defines the functors
$\BB_\ast : \C\Module \to \C_+\Module$, which we shall also regard as functors to $\C\Module$  by regarding $\C_+\Module$ as a subcategory of $\C\Module$ in the obvious way. We call $\BB_\ast$ the negative shift functor.

On the other hand, the self-submerging functor $\boldsymbol{\pi}_{\ast}: \C_+ \to \C$ also induces a functor as follows via pull-back:
\[
\BB'_{\ast}: \C \Module \to \C_+ \Module, \quad V \to V \circ \boldsymbol{\pi}_{\ast}.
\]
Again, we can view $\BB'_{\ast}$ as an endo-functor in $\C \Module$. Furthermore, the natural transformation $\boldsymbol{\rho}^{\ast}: \boldsymbol{\pi}_{\ast} \to \ID_{\C}$ induces a natural module homomorphism $\BB'_{\ast} V \to V$ for each $V \in \C \Module$.

In the language of category algebras, there is a surjective homomorphism $\varpi_a: k\C_+\to k\C$ defined by $\varpi_a(\alpha)=\alpha_a$ if $\alpha(1)=1$, otherwise $\varpi_a(\alpha)=0$. The kernel of $\varpi_a$ is $\mi_a$, so we have an isomorphism $k\C_+/\mi_a \cong k\C$. Similarly, there is an isomorphism $k\C_+/\mi_b \cong k\C$. Thus the functor $\BB_\ast$ is the restriction along $k\C_+ \to k\C_+/\mi_* \cong k\C$. Analogously, the functor $\boldsymbol{\pi}_{\ast}$ also induces a surjective algebra homomorphism $k\C_+ \to k\C$, and $\BB'_{\ast}$ is the restriction along it. Consequently, we obtain an isomorphism $\BB_{\ast} \cong \BB'_{\ast}$.

\begin{lemma}
Notations as before. Then one has:
\begin{enumerate}
\item $\BS_{\ast} \circ \BB_{\ast}' \cong \ID_{\C \Module}$.
\item The kernel of the algebra homomorphism $k\C_+ \to k\C$ induced by $\boldsymbol{\pi}_a$ (resp., $\boldsymbol{\pi}_b$) is spanned by $\langle \alpha_{m, 1} - \alpha_{m, 2} \mid m \geqslant 1 \rangle$ (resp., $\langle \alpha_{m, m} - \alpha_{m, m+1} \rangle$) as a left $k\C$-ideal.
\end{enumerate}
\end{lemma}

\begin{proof}
We prove the conclusion for $\ast = a$, and the proof of the other case is similar.

(1): This follows immediately from the fact that $\boldsymbol{\pi}_a \circ \boldsymbol{\iota}_a \cong \ID_{\C}$.

(2): Using essentially the same proof as that of \cite[Lemma 6.1]{Li}, one can show that the kernel, as a $k$-module, is spanned by elements of the form $\alpha - \beta$, where $\alpha$ and $\beta$ are morphisms sharing the same source $[m]$ and target $[n]$ such that $\alpha$ and $\beta$ restricted to $[m] \setminus \{1\}$ define the same map. Clearly, $\alpha_{m, 1} - \alpha_{m, 2}$ is contained in the kernel, so the left ideal generated by these elements is contained in the kernel.

To show the inclusion of the other direction, let us consider the above mentioned $\alpha - \beta$. Clearly, we can assume that $\alpha \neq \beta$, and hence $\alpha(1) \neq \beta(1)$. Without loss of generality suppose that $\alpha(1) > \beta(1)$. Then one can write $\alpha = \gamma \circ \alpha_{m, 1}$ and $\beta = \gamma \circ \alpha_{m, 2}$, where $\gamma: [m+1] \to [n]$ is defined by $\gamma(1) = \beta(1)$, $\gamma(2) = \alpha(1)$, and $\gamma(i) = \alpha(i) = \beta(i)$ for $i \geqslant 3$. Thus $\alpha - \beta$ is contained in the left ideal generated by the set $\{ \alpha_{m, 1} - \alpha_{m,2} \mid m \geqslant 1\}$ as claimed.
\end{proof}

\begin{proposition}  \label{prop:B triple}
Notation as before. Then one has:
\begin{enumerate}
\item $\BB_{\ast} M(n) \cong M(n+1)/\mi_{\ast} M(n+1) = M(n+1)/ \mi_{\ast} e_{n+1}$.
\item $(\BD_{\ast}, \, \BB_{\ast}, \, \BS_{\ast} \BK_{\ast})$ is a triple of adjoint functors.
\end{enumerate}
\end{proposition}

\begin{proof}
We prove the proposition for $\ast = a$. The first statement holds by a clear isomorphism
\[
\C([n], [n+r]) \ni \alpha \mapsto f \in \C([n+1], [n+1+r]
\]
where we identity $f$ with its image in the quotient algebra, and let $f(1) = 1$, $f(i) = \alpha(i-1) + 1$ for $i \in [n+1] \setminus \{1\}$.

To show that the right adjoint of $\BB_a$ is $\BS_a \BK_a$, we have
\begin{align*}
\Hom_{k\C_+} (\BB_a V, W) & \cong \Hom_{k\C_+}(k\C_+/\mi_a \otimes_{k\C} V, W)\\
 & \cong \Hom_{k\C} (V, \bigoplus_{n \geqslant 1} \Hom_{k\C_+}(M(n)/\mi_a M(n), W))\\
 & \cong \Hom_{k\C} (V, \mathrm{Ann}_{\mi_a}(W))\\
 & \cong \Hom_{k\C} (V, \bigoplus_{n \geqslant 1} (\BK_a W)_n)
\end{align*}
where the second isomorphism comes from the tensor-hom adjunction which also applies to our case (see \cite[Definition 4.1 and Lemma 4.2]{GL1}). Note that the $k\C_+$-module $\bigoplus_{n \geqslant 1} (\BK_a V)_n$ is viewed as a $k\C$-module via the isomorphism $k\C \cong k\C_+/\mi_a$, and hence is precisely $\BS_a \BK_aV$ as a $\C$-module. Therefore, $\BS_a \BK_a$ is the right adjoint of $\BB_a$.

The right adjoint of $\BD_a$ on a $\C$-module $V$ is the $\C_+$-module whose value on $[n]$, $n \geqslant 1$, is:
\[ \Hom_{k\C} (\BD_a (M(n)), V ) \cong \Hom_{k\C} (M(n-1), V) \cong V_{n-1}. \]
If $\alpha: [n] \to [n']$ is a morphism in $\C$, then we get a morphism $\tilde{\alpha}: M(n') \to M(n)$ where $\beta \mapsto \beta \alpha$, and hence a morphism $\BD_a(\tilde{\alpha}): \BD_a M(n') \to \BD_a M(n)$. If we interpret this as a morphism $M(n'-1) \to M(n-1)$, then it is zero if $\alpha(1) > 1$, otherwise if $\alpha(1)=1$ it is the morphism $\beta \mapsto \beta \alpha'$ where
\[
\alpha': [n-1] \to [n'-1] , \, i \mapsto \alpha(i+1)-1.\]
Thus the induced map
\[
 \Hom_{k\C} ( \BD_a (M(n)), V ) \to \Hom_{k\C} (\BD_a (M(n')), V ) \]
is identified with the map $V_{n-1} \to V_{n'-1}$ where $v \mapsto \alpha' v$. It follows that the right adjoint of $\BD_a$ is $\BB_a$.
\end{proof}

\begin{corollary}
One has $\BD_{\ast} \boldsymbol{\Gamma}_{\ast} \cong \ID_{\C \Module}$.
\end{corollary}

\begin{proof}
Let $V$ and $W$ be two $\C$-modules. Since $\BS_{\ast} \BB_{\ast} W \cong W$, one shall have
\begin{align*}
\Hom_{k\C} (V, W) & \cong \Hom_{k\C} (V, \BS_{\ast} \BB_{\ast}W)\\
 & \cong \Hom_{k\C_+} (\boldsymbol{\Gamma}_{\ast} V, \BB_{\ast}W) \\
 & \cong \Hom_{k\C} (\BD_{\ast} \boldsymbol{\Gamma}_{\ast} V, W),
\end{align*}
that is, $\BD_{\ast} \boldsymbol{\Gamma}_{\ast}$ is the left adjoint of $\ID_{\C \Module}$. The conclusion then follows from the uniqueness of left adjoint functors.
\end{proof}

\subsection{The right adjoint of $\BS_{\ast}$}
By Section \ref{subsection:right adjoint}, the shift functor $\BS_{\ast}$ has a right adjoint functor $\BS_\ast^\dag$. Below, we prove that $\BS_\ast^\dag$ is isomorphic to a functor $\BQ_\ast$ which has a more explicit description. (The functor $\BQ_{\ast}$ is essentially the same as the coinduction functor defined in \cite[Subsection 13.1]{GS}.)

For any morphism $\alpha: [n]\to [\ell]$ such that $\alpha(1)>1$, define a morphism $\partial_a \alpha : [n]\to [\ell-1]$ by $\partial_a\alpha (i) = \alpha(i)-1$ for each $i\in [n]$. Similarly, for any morphism $\alpha: [n]\to [\ell]$
such that $\alpha(n)<\ell$, define a morphism $\partial_b \alpha : [n]\to [\ell-1]$ by $\partial_b\alpha (i) = \alpha(i)$ for each $i\in [n]$.

For any $\C$-module $V$, we define a $\C$-module $\BQ_\ast V$ in the following way. For each $n\in \N$, let
\[ (\BQ_\ast V)_n = V_n \oplus V_{n-1}, \]
where we set $V_{-1}=0$.
For each morphism $\alpha: [n]\to [\ell]$ in $\C$, we define the induced map $(\BQ_a V)_n \to (\BQ_a V)_\ell$ to be the map
\begin{align*}
V_n \oplus V_{n-1} &\to V_\ell \oplus V_{\ell-1},\\
(v, w) &\mapsto
\left\{  \begin{array}{ll}
(\alpha(v),\, \alpha_a (w)  )  & \mbox{ if } \alpha(1)=1,\\
(\alpha(v),\, \partial_a \alpha (v))  & \mbox{ if } \alpha(1)>1;
\end{array} \right.
\end{align*}
similarly, we define the induced map $(\BQ_b V)_n \to (\BQ_b V)_\ell$ to be the map
\begin{align*}
V_n \oplus V_{n-1} &\to V_\ell \oplus V_{\ell-1},\\
(v, w) &\mapsto
\left\{  \begin{array}{ll}
(\alpha(v),\, \alpha_b (w)  )  & \mbox{ if } \alpha(n)=\ell,\\
(\alpha(v),\, \partial_b \alpha (v))  & \mbox{ if } \alpha(n)<\ell.
\end{array} \right.
\end{align*}
We remind the reader that $\alpha_a : [n-1]\to [\ell-1]$
and $\alpha_b : [n-1]\to [\ell-1]$
are defined by $\alpha_a (i)=\alpha(i+1)-1$
and $\alpha_b(i) = \alpha(i)$
for each $i\in [n-1]$.

It can be checked directly that the above construction defines a $\C$-module $\BQ_* V$ and hence a functor $\BQ_*: \C\Module \to \C\Module$, but this also follows from the proof of the following proposition which identifies $\BQ_*$ with the right adjoint functor $\BS_*^\dag$ of $\BS_*$.

\begin{proposition} \label{prop:Q}
One has:
\begin{enumerate}
\item $\BQ_{\ast}$ is the right adjoint of $\BS_{\ast}$.
\item For any $\C$-module $V$, there is a short exact sequence
\[
0 \to \BB_{\ast} V \to \BQ_{\ast} V \to V \to 0.
\]
In particular, one has $\BQ_{\ast} M(n) \cong M(n) \oplus M(n+1)/\mi_{\ast}M(n+1)$.

\item $\BB_{\ast} \cong \BK_{\ast} \BQ_{\ast}$.
\end{enumerate}
\end{proposition}

\begin{proof}
We prove the proposition for $\ast = a$.

(1): Since $\BS_a$ is right exact and transforms direct sums to direct sums, it has a right adjoint functor $\BS_a^\dag$ by Section \ref{subsection:right adjoint}. We show that $\BQ_a$ is isomorphic to $\BS_a^\dag$.

First, for each $n\in\N$, one has an isomorphism
\[ f: \BS_a M(n) \to M(n) \oplus M(n-1), \]
where we set $M(-1)=0$; if $\alpha: [n]\to [\ell+1]$ is a morphism in $(\BS_a M(n))_\ell$, then
\[ f(\alpha) = \left\{
\begin{array}{ll}
\alpha_a \in M(n-1)_\ell & \mbox{ if } \alpha(1)=1,\\
\partial_a \alpha \in M(n)_\ell & \mbox{ if } \alpha(1)>1.
\end{array}
\right. \]

Next, for each $n\in\N$, one has a $k$-module isomorphism
\[  g: V_n \oplus V_{n-1} \to \Hom_{k\C} (M(n)\oplus M(n-1), V ) \]
defined by $g: (v,\, w) \mapsto \phi$ where $\phi(\alpha_1, \alpha_2) = \alpha_1(v) + \alpha_2(w)$ for any $\alpha_1 \in M(n)_r$ and $\alpha_2 \in M(n-1)_r$, for any $r$.

Therefore we have the $k$-module isomorphisms
\[ (\BQ_a V)_n = V_n \oplus V_{n-1} \cong \Hom_{k\C} (M(n)\oplus M(n-1), V )
\cong \Hom_{k\C} (\BS_a M(n), V) = (\BS_a^\dag V)_n. \]
We claim that these isomorphisms are compatible with the $\C$-module structures of $\BQ_a V$ and $\BS_a^\dag V$.

Let $\alpha: [n]\to [\ell]$ be any morphism in $\C$. Following Section \ref{subsection:right adjoint}, we have $\rho_\alpha: M(\ell) \to M(n)$ and hence the homomorphism
$\BS_a (\rho_\alpha) : \BS_a M(\ell) \to \BS_a M(n)$, which we identify with the homomorphism
\begin{align*}
h: M(\ell)\oplus M(\ell-1) &\to M(n)\oplus M(n-1), \\
(\beta_1, \, \beta_2) &\mapsto  \left\{ \begin{array}{ll}
(\beta_1 \circ \alpha, \, \beta_2 \circ \alpha_a ) & \mbox{ if } \alpha(1)=1, \\
(\beta_1 \circ \alpha + \beta_2 \circ \partial_a \alpha , \, 0)& \mbox{ if } \alpha(1)>1.
\end{array}\right.
\end{align*}
Now consider any $(v,\, w)\in V_n \oplus V_{n-1}$ and let $\phi=g(v,w)$ as above, so $\phi : M(n)\oplus M(n-1) \to V$ sends $(\alpha_1, \alpha_2)$ to $\alpha_1(v)+\alpha_2(w)$. The composition of the above homomorphism $h$ with $\phi$ is
\[ \phi(h(\beta_1, \beta_2)) = \left\{ \begin{array}{ll}
\beta_1 (\alpha(v)) + \beta_2 (\alpha_a (w))  & \mbox{ if } \alpha(1)=1, \\
\beta_1 (\alpha(v)) + \beta_2 (\partial_a \alpha (v)) & \mbox{ if } \alpha(1)>1.
\end{array}\right.
 \]
It follows that the $\C$-modules $\BQ_a V$ and $\BS_a^\dag V$ are isomorphic as claimed.

(2): Define a homomorphism $\BB_\ast V \to \BQ_\ast V$ by sending
\[ w\in (\BB_\ast V)_n = V_{n-1} \qquad \mbox{ to } \qquad (0, w) \in V_n \oplus V_{n-1} = (\BQ_\ast V)_n. \]
Define a homomorphism $\BQ_\ast V \to V$ by sending
\[ (v, w) \in V_n \oplus V_{n-1} = (\BQ_\ast V)_n \qquad \mbox{ to } \qquad v \in V_n.\]
It is clear that these give a short exact sequence $0\to\BB_\ast V \to \BQ_\ast V \to V\to 0 $.

By Proposition \ref{prop:B triple}, we have
$\BB_\ast M(n) \cong M(n+1)/\mi_{\ast}M(n+1)$. Since $M(n)$ is a projective $\C$-module, it follows from the short exact sequence that we have
$\BQ_{\ast} M(n) \cong M(n) \oplus M(n+1)/\mi_{\ast}M(n+1)$.

(3): Let $\alpha:[n]\to [n+1]$ be the map $\alpha(i)=i+1$ for every $i\in [n]$.
Then $\partial_a \alpha$ is the identity map on $[n]$. Hence, for any
$(v,w)\in V_n \oplus V_{n-1} = (\BQ_a V)_n$, we have $\alpha(v,w)=(\alpha(v), v)$; therefore $\alpha(v,w)=0$ if and only if $v=0$. It follows that the image of the above injective map $\BB_a V \to \BQ_a V$ is precisely $\BK_a \BQ_a V$, hence $\BB_a V \cong\BK_a \BQ_a V $.
\end{proof}

One can also prove $\BK_{\ast} \BQ_{\ast} \cong \BB_{\ast}$ using the tensor hom interpretations of these functors. Firstly, note that
\[
\BK_{\ast} \cong \mathrm{Ann}_{\mi_{\ast}}(-) \cong \bigoplus_{n \geqslant 1} \Hom_{k\C} (M(n)/\mi_{\ast} M(n), -).
\]
Therefore, via the tensor-hom adjunction (which also holds in our situation via a suitable modification), we see that $(k\C/\mi_{\ast} \otimes_{k\C} -, \BK_{\ast})$ is an adjoint pair. But we also know that $\BD_{\ast} \cong \BS_{\ast} (k\C/\mi_{\ast} \otimes_{k\C} -)$. Thus the right adjoint of $\BD_{\ast}$ is $\BK_{\ast} \BQ_{\ast}$. By Proposition \ref{prop:B triple}, we have $\BK_{\ast} \BQ_{\ast} \cong \BB_{\ast}$.

\subsection{Right adjoint of $\BQ_{\ast}$}
For any morphism $\alpha: [n]\to [\ell]$ such that $\alpha(1)>1$, define a morphism
$\partial^a \alpha: [n+1]\to [\ell]$ by
\[ \partial^a \alpha(i) = \left\{ \begin{array}{ll}
1  & \mbox{ if } i=1,\\
\alpha(i-1) & \mbox{ if }i>1.
\end{array}\right. \]
Similarly, if $\alpha:[n]\to [\ell]$ and $\alpha(n)<\ell$, define $\partial^b \alpha: [n+1]\to \ell$ by
\[ \partial^b \alpha(i) = \left\{ \begin{array}{ll}
\alpha(i)  & \mbox{ if } i\leqslant n,\\
\ell & \mbox{ if }i=n+1.
\end{array}\right. \]

For any $\C$-module $V$, we define a $\C$-module $\BR_\ast V$ in the following way. For each $n\in \N$, let
\[ (\BR_\ast V)_n = V_n \oplus (\BK_* V)_{n+1}.  \]
Note that $(\BK_* V)_{n+1}\subseteq V_{n+1}$.
For each morphism $\alpha: [n]\to [\ell]$ in $\C$, we define the induced map
$(\BR_a V)_n \to (\BR_a V)_\ell$ to be the map
\begin{align*}
V_n \oplus (\BK_a V)_{n+1} &\to V_\ell \oplus (\BK_a V)_{\ell+1},\\
(v, w) &\mapsto   \left\{ \begin{array}{ll}
 (\alpha(v),\, \boldsymbol{\iota}_a (\alpha)(w)) & \mbox{ if } \alpha(1)=1,\\
 (\alpha(v) - \partial^a \alpha (w),\, \boldsymbol{\iota}_a (\alpha)(w) ) & \mbox{ if } \alpha(1)>1;
\end{array}\right.
\end{align*}
similarly for the induced map $(\BR_b V)_n \to (\BR_b V)_\ell$.

It can be checked directly that the above construction defines a $\C$-module $\BR_* V$ and hence a functor $\BR_*: \C\Module \to \C\Module$, but this also follows from the proof of the following proposition which identifies $\BR_*$ with the right adjoint functor $\BQ_*^\dag$ of $\BQ_*$.

\begin{proposition}
Notations as above. Then one has:
\begin{enumerate}
\item $\BR_{\ast}$ is the right adjoint functor of $\BQ_{\ast}$.
\item There is a short exact sequence $0 \to V \to \BR_{\ast}V \to \BS_{\ast} \BK_{\ast}V \to 0$.
\end{enumerate}
\end{proposition}

\begin{proof}
We prove the proposition for $\ast = a$.

(1): Since $\BQ_a$ is right exact and transforms direct sums to direct sums, it has a right adjoint functor $\BR_a^\dag$ by Section \ref{subsection:right adjoint}. We show that $\BR_a$ is isomorphic to $\BQ_a^\dag$.

Let $V$ be a $\C$-module.
By Proposition \ref{prop:Q}(2), we have an isomorphism
\[ \BQ_a M(n) \cong M(n) \oplus M(n+1)/\mi_a M(n+1), \]
hence we have an isomorphism
\[ (\BQ_a^\dag V)_n = \Hom_{k\C}( \BQ_a M(n), V) \cong V_n \oplus (\BK_a V)_{n+1} = (\BR_a V)_n. \]
Let us choose this isomorphism so that a homomorphism $\phi: \BQ_a M(n)\to V$ is identified with
\[ (\phi(e_n,0),\, \phi(0, e_n)) \in V_n \oplus (\BK_a V)_{n+1}, \]
where
\begin{gather*}
(e_n,0) \in M(n)_n \oplus M(n)_{n-1} = (\BQ_a M(n))_n,\\
(0, e_n) \in M(n)_{n+1} \oplus M(n)_n = (\BQ_a M(n))_{n+1}.
\end{gather*}
We claim that this identification is compatible with the $\C$-module structures of $\BQ_a^\dag V$ and $\BR_a V$.

Let $\alpha: [n]\to [\ell]$ be any morphism in $\C$. Following Section \ref{subsection:right adjoint}, we have $\rho_\alpha: M(\ell) \to M(n)$ and hence the homomorphism
$\BQ_a (\rho_\alpha) : \BQ_a M(\ell) \to \BQ_a M(n)$. The homomorphism $\BQ_a (\rho_\alpha)$ sends
\begin{align*}
(e_\ell, 0)\in M(\ell)_\ell \oplus M(\ell)_{\ell-1} \quad&\mbox{ to }\quad
(\alpha, 0 ) \in M(n)_\ell \oplus M(n)_{\ell-1},\\
(0, e_\ell)\in M(\ell)_{\ell+1} \oplus M(\ell)_{\ell} \quad&\mbox{ to }\quad
(0, \alpha) \in M(n)_{\ell+1} \oplus M(n)_{\ell} .
\end{align*}
From the definition of the map $(\BQ_a M(n))_n \to (\BQ_a M(n))_\ell$ induced by $\alpha$, we have:
\[ (\alpha, 0 ) = \left\{ \begin{array}{ll}
\alpha (e_n, 0) & \mbox{ if } \alpha(1)=1,\\
\alpha (e_n, 0) - \partial^a \alpha(0, e_n) & \mbox{ if } \alpha(1)>1.
\end{array}\right. \]
From the definition of the map $(\BQ_a M(n))_{n+1} \to (\BQ_a M(n))_{\ell+1}$ induced by $\boldsymbol{\iota}_a (\alpha)$, we have:
\[
(0, \alpha) = \boldsymbol{\iota}_a (\alpha) (0, e_n).
\]

Consider any $(v, w)\in V_n \oplus (\BK_a V)_{n+1}$. Let  $\phi: \BQ_a M(n)\to V$ be the homomorphism such that $\phi(e_n,0)=v$ and $\phi(0,e_n)=w$. Then we have:
\[ (\phi \circ \BQ_a (\rho_\alpha)) (e_\ell, 0) =  \left\{ \begin{array}{ll}
 \alpha(v) & \mbox{ if } \alpha(1)=1,\\
 \alpha(v) - \partial^a \alpha (w) & \mbox{ if } \alpha(1)>1;
\end{array}\right. \]
\[  (\phi \circ \BQ_a (\rho_\alpha)) (0, e_\ell) = \boldsymbol{\iota}_a (\alpha)(w). \]
It follows that the $\C$-modules $\BR_a V$ and $\BQ_a^\dag V$ are isomorphic as claimed.

(2): This is immediate from the definition of $\BR_*$.
\end{proof}

\section{The Nakayama functor}

In this section we describe a torsion theory of $\C$-modules, and prove the second main result of this paper.

\subsection{Torsion theory}

Recall that for an $\C$-module $V$, an element $v \in V_n$ is \textbf{torsion} if there exists a morphism $\alpha: [n] \to [n+s]$ such that $\alpha \cdot v = 0$. The module $V$ is \textbf{torison} if for every $n \in \mathbb{N}$, every $v \in V_n$ is a torsion element. The module $V$ is \textbf{torsion free} if it contains no nonzero torsion elements. In particular, every projective $\C$-module is torsion free.

We shall point out that the definition of torsion modules in this paper is completely different from that defined in \cite[Subsction 9.1]{GS}, where an element $v \in V_n$ is torsion if and only if $\mi_b \cdot v = 0$.

\begin{lemma}
For $V \in \C \Module$,
\[
\bigoplus_{n \in \mathbb{N}} \{v \in V_n \mid v \text{ is torison} \}
\]
is a submodule of $V$.
\end{lemma}

\begin{proof}
We need to check the following two facts:
\begin{enumerate}
\item if $v \in V_n$ is torsion and $\alpha: [n] \to [n+s]$ is a morphism, then $\alpha \cdot v$ is also a torsion element;
\item if $v \in V_n$ and $w \in V_n$ are torsion, so is $v + w$.
\end{enumerate}
Their proofs rely on the following easy observation: given two morphisms $f_1 \in \C([m], [n])$ and $f_2 \in \C([m], [r])$, there exist an object $[s]$ and morphisms $g_1 \in \C([n], [s])$ and $g_2 \in \C([r], [s])$ such that $g_1 f_1 = g_2 f_2$.

Since $v$ is torsion, one can find a certain $f: [n] \to [n+r]$ such that $f \cdot v = 0$. By the observation, we can take an object $[q]$ and morphisms $\beta_1$ and $\beta_2$ ending at $[q]$ with $\beta_1 f = \beta_2 \alpha$. Consequently, $\beta_2 (\alpha \cdot v) = \beta_1 (f \cdot v) = 0$, so $\alpha \cdot v$ is torsion. Similarly, if $v$ and $w$ are torsion, there exist $\alpha_1$ and $\alpha_2$ starting at $[n]$ such that $\alpha_1 \cdot v = \alpha_2 \cdot w$ = 0. Now choose $\beta_1$ and $\beta_2$ with $\beta_1 \alpha_1 = \beta_2 \alpha_2$. Then
\[
(\beta_1 \alpha_1) \cdot (v+w) = (\beta_1 \alpha_1) \cdot v + (\beta_2 \alpha_2)\cdot w = 0.
\]
This finishes the proof.
\end{proof}

The submodule constructed in the above lemma is the maximal torsion submodule of $V$, denoted by $V_T$. The quotient $V_F = V/V_T$ is called the torsion free part of $V$. Let $\C \Module^{\tor}$ and $\C \Module^{\rm{tf}}$ be the categories of torsion $\C$-modules and torsion free $\C$-modules respectively. It is easy to check that $\C \Module^{\tor}$ is an abelian category, and
\[
\Hom_{k\C} (T, F) = 0, \quad \forall T \in \C \Module^{\tor}, \, \forall F \in \C \Module^{\rm{tf}}.
\]
Furthermore, if $\Hom_{k\C} (T, F) = 0$ for all torsion modules $T$, then $F \in \C \Module^{\rm{tf}}$; dually, if $\Hom_{k\C} (T, F) = 0$ for all torsion free modules $F$, then $T \in \C \Module^{\tor}$.

\subsection{Serre quotient}

Recall that $\C \Module^{\tor}$ is abelian, we obtain a Serre quotient category $\C \Module / \C \Module^{\tor}$ and two functors
\[
\xymatrix{
\C \Module \ar@<.5ex>[rr]^-{\loc} & & \C \Module / \C \Module^{\tor} \ar@<.5ex>[ll]^-{\se},
}
\]
where $\loc$ and $\se$ are called the localization functor and the section functor respectively.

A $\C$-module $V$ is \textbf{saturated} if one of the following equivalent conditions holds.

\begin{lemma} \label{saturated modules}
Let $V$ be a $\C$-module. Then the following conditions are equivalent.
\begin{enumerate}
\item $(\se \circ \loc) V \cong V$;
\item $\Hom_{\C} (T, V) = 0 = \Ext_{k\C}^1 (T, V)$ for any torsion $\C$-module $T$;
\item there is a homomorphism $f: I^1 \to I^0$ between torsion free injective modules such that $V \cong \ker f$.
\end{enumerate}
\end{lemma}

\begin{proof}
The equivalence between (1) and (2) is a classical result of Gabriel; see \cite{Gab}. The equivalence between (2) and (3) is also straightforward by noting that torsion free injective $\C$-modules are torsion acyclic.
\end{proof}

It follows from \cite{Gab} that the functor $\se$ sends an object in $\C \Module /\C \Module^{\tor}$ to a saturated module. Therefore, if we let $\C \Module^{\rm{sa}}$ be the full subcategory consisting of all saturated modules, then $\sec$ and $\loc$ give an equivalence between $\C \Module^{\rm{sa}}$ and $\C \Module / \C \Module ^{\tor}$. However, the abelian structure of $\C \Module^{\rm{sa}}$ is not inherited from that of $\C \Module$. More explicitly, let $\alpha: V \to W$ be a morphism in $\C \Module^{\rm{sa}}$. Then its kernel in $\C \Module^{\rm{sa}}$ is the same as its kernel in $\C \Module$, but the cokernels of are different. Actually, the cokernel of $\alpha$ in $\C \Module^{\rm{sa}}$ is defined to be $(\rm{sec} \circ \rm{loc}) (W/\alpha(V))$.

\subsection{An equivalence}

From now on we let $k$ be a field. Denote the category of finitely generated $\C$-modules by $\C \module$ (which is abelian), and the category of finitely generated torsion $\C$-modules by $\C \module^{\tor}$. Restricted to this framework, we have another pair of functors
\[
\xymatrix{
\C \module \ar@<.5ex>[rr]^-{\loc} & & \C \module / \C \module^{\tor} \ar@<.5ex>[ll]^-{\se},
}
\]
and the conclusion of Lemma \ref{saturated modules} also holds for this new setup (note that $\C \module$ also has enough injectives, see \cite[Theorem 14.2]{GS}). We want to describe an equivalence between $\C \module / \C \module^{\tor}$ and $\C \fdmod$, the category of finite dimensional $\C$-modules, induced by the Nakayama functors. Since these functors have been studied in \cite{GLX} for representations of discrete categories satisfying certain combinatorial conditions, in this paper we only briefly mention their definitions. For more details, please refer to \cite[Section 2]{GLX}.

Recall that the Nakayama functor $\boldsymbol{\nu}: \C \module \to \C \fdmod$ is defined to be
\[
\boldsymbol{\nu} = \Hom_k(-, k) \circ \Hom_{k\C} (-, k\C) = \Hom_k(-, k) \circ \bigoplus_{n \in \N} \Hom_{k\C} (-, k\C e_n).
\]
Since for $V \in \C \module$, there are only finitely many $n \in \N$ such that $\Hom_{k\C} (V, k\C e_n)) \neq 0$, $\boldsymbol{\nu} V$ is finite dimensional. Thus $\boldsymbol{\nu}$ is well defined. Furthermore, it sends finitely generated projective $\C$-modules to finite dimensional injective $\C$-modules.

The inverse Nakayama functor $\boldsymbol{\nu}^{-1}: \C \fdmod \to \C \module$ is defined to be
\[
\boldsymbol{\nu}^{-1} = \Hom_{(k\C)^{\mathrm{op}}} (-, k\C) \circ \Hom_k(-, k) = \bigoplus_{n \in \N} \Hom_{(k\C)^{\mathrm{op}}} (-, e_n k\C) \circ \Hom_k(-, k).
\]
By \cite[Lemma 2.2]{GLX}, the functor $\boldsymbol{\nu}^{-1}$ is also well defined.

To prove the second theorem in the Introduction, we need some elementary facts about injective $\C$-modules and Nakayama functors.

\begin{proposition}\cite[Theorem 14.2]{GS} \label{injectives}
Let $V$ be a finitely generated torsion free $\C$-module. Then $V$ is injective if and only if it is projective.
\end{proposition}

\begin{proof}
By \cite[Proposition 3.7]{GS}, one can convert representation theory of $\OI$ to representation theory of the increasing monoid. If $V$ is projective, then it is also injective by \cite[Theorem 10.12]{GS}. Conversely, the classification of injective objects has been described in \cite[Theorem 14.2]{GS}, and in particular, torsion free injective objects are projective modules; more explicitly, every indecomposable torsion free injective $\C$-module is isomorphic to $M(n)$ for a certain $n \in \N$.
\end{proof}

\begin{proposition}
Notation as above. Then we have:
\begin{enumerate}
\item The pair $(\boldsymbol{\nu}, \, \boldsymbol{\nu}^{-1})$ is an adjoint pair.
\item The functor $\boldsymbol{\nu}$ is exact.
\item $\boldsymbol{\nu} \circ \boldsymbol{\nu}^{-1} \cong \ID_{\C \fdmod}$.
\end{enumerate}
\end{proposition}

\begin{proof}
Statements (1) follows from \cite[Lemma 2.4]{GLX}. Statement (2) follows from \cite[Lemma 2.11]{GLX} since finitely generated projective $\C$-modules are also injective. Statement (3) is \cite[Proposition 3.3]{GLX}.
\end{proof}

Now we are ready to prove the following result.

\begin{theorem}
We have the following equivalence of abelian categories
\[
\xymatrix{
\C \module / \C \module^{\tor} \ar@<.5ex>[rr]^-{\boldsymbol{\nu}} & & \C \fdmod \ar@<.5ex>[ll]^-{\boldsymbol{\nu}^{-1}}.
}
\]
\end{theorem}

\begin{proof}
By \cite[Proposition III.2.5]{Gab}, the Nakayama functor $\boldsymbol{\nu}$ induces an equivalence $\bar{\boldsymbol{\nu}}: \C \module / \ker \boldsymbol{\nu} \to \C \fdmod$, so we only need to show that $\ker \boldsymbol{\nu} = \C \module^{\tor}$. Clearly, $\C \module^{\tor}$ is contained in the kernel, so it suffices to show the inclusion of the other direction, or equivalently, any finitely generated $\C$-module which is not torsion is not sent to 0 by $\boldsymbol{\nu}$.

Suppose that $V$ is not torsion. Applying $\boldsymbol{\nu}$ to the short exact sequence $0 \to V_T \to V \to V_F \to 0$ one gets $\boldsymbol{\nu} V \cong \boldsymbol{\nu} V_F$ since $\boldsymbol{\nu}$ is exact and $\boldsymbol{\nu} V_T = 0$. Therefore, we may assume that $V = V_F$. Consider an injection $0 \to V \to I$ with $I$ injective. We can extend it to a commutative diagram of exact sequences
\[
\xymatrix{
0 \ar[r] & V \cap I_T \ar[r] \ar[d] & I_T \ar[r] \ar[d] & I_T/(V \cap I_T) \ar[r] \ar[d] & 0\\
0 \ar[r] & V \ar[r] \ar[d] & I \ar[r] \ar[d] & I/V \ar[r] \ar[d] & 0\\
0 \ar[r] & V/(V\cap I_T) \ar[r] & I_F \ar[r] & I_F/(V /(V \cap I_T)) \ar[r] & 0.
}
\]
Since $V$ is torsion free and $I_T$ is torsion, $V \cap I_T = 0$. Therefore, $V$ is actually isomorphic to a submodule of $I_F$. Note that the middle column splits: since the functor $V \mapsto V_T$ is the right adjoint of the exact inclusion $\C \module^{\tor} \to \C \module$, it preserves injectives, and hence $I_T$ is injective as well. Therefore, $I_F$ is a finitely generated torsion free injective $\C$-module, which must be projective by Proposition \ref{injectives}. Consequently, the existence of the injective map $V \to I_F$ tells us that $\boldsymbol{\nu} V \neq 0$. This finishes the proof.
\end{proof}

\begin{remark} \normalfont
We remind the reader that this result is different from Statement (d) of \cite[Proposition 9.1]{GS}, since our definition of torsion modules is completely different. Therefore, neither of our result nor their result imply the other one.
\end{remark}

An immediate corollary of this theorem is:

\begin{corollary}
Every object in $\C \module^{\rm{sa}}$ is of finite length and of finite injective dimension.
\end{corollary}

\begin{proof}
This is clear since $\C \module^{\rm{sa}}$ is equivalent to $\C \fdmod$.
\end{proof}

As a conclusion, we obtain the following commutative diagram of functors:
\[
\xymatrix{
\C \module^{\rm{sa}} \ar[r]^-{\mathrm{inc}} & \C \module \ar@<.5ex>[r]^-{\boldsymbol{\nu}} \ar[d]^-{\mathrm{loc}} & \C \fdmod \ar@<.5ex>[l]^-{\boldsymbol{\nu}^{-1}} \\
 & \C \module / \C \module^{\tor} \ar[ul]^-{\mathrm{sec}} \ar[ur]_-{\sim}
}
\]
and the following equivalences of abelian categories:
\[
\xymatrix{
\C \module^{\rm{sa}} \ar[r]^-{\sim} & \C \module / \C \module^{\tor} \ar[r]^-{\sim} & \C \fdmod.
}
\]

\subsection{Simple saturated modules}

One important application of the above equivalences is the classification of simple objects in $\C \Module^{\rm{sa}}$. Clearly, they are contained in $\C \module^{\rm{sa}}$ and correspond to simple objects in $\C \fdmod$. But simple objects in $\C \fdmod$ are parameterized by $\N$; that is, for every $n \in \N$, there is a unique simple $\C$-module whose value on $[n]$ is $k$, and whose values on other objects are 0. Therefore, isomorphism classes of simple saturated $\C$-modules are parameterized by $\N$ as well. Furthermore, from the above equivalence we also know that every simple saturated module is the socle of an indecomposable injective torsion free module, which is isomorphic to $M(n)$ for a certain $n \in \N$.

It is easy to check that $M(0)$ is a simple saturated module. Now suppose that $n \geqslant 1$. Note that there is a surjective homomorphism
\[
k\C e_n \to k\C \alpha_{n-1,i}, \quad e_n \mapsto \alpha_{n-1, i}
\]
denoted by $f_i$. Let $K_i$ be the kernel of $f_i$.

\begin{lemma}
Notation as above. Then $K_i$ as a two-sided ideal of $k\C$ is generated by $\langle \alpha_{n, i} - \alpha_{n, i+1} \rangle$.
\end{lemma}

\begin{proof}
By the definition of $f_i$, one has
\[
f_i(\alpha_{n, i} - \alpha_{n, i+1}) = \alpha_{n, i} \alpha_{n-1, i} - \alpha_{n, i+1} \alpha_{n-1, i} = 0
\]
Therefore, the two-sided ideal generated by $\langle \alpha_{n, i} - \alpha_{n, i+1} \rangle$ is contained in $K_i$. Conversely, by the definition of $f_i$, and using the same technology of \cite[Lemma 6.1]{Li}, one can show that $K_i$ is spanned (as a $k$-module) by elements $\beta - \gamma$ such that $\beta$ and $\gamma$ share the same source and target, and $\beta \alpha_{n-1, i} = \gamma \alpha_{n-1, i}$. Suppose that $\gamma > \beta$; that is, $\gamma \neq \beta$ and $\gamma(i) \geqslant \beta(i)$ for $i \in [n]$. Then one can find a morphism $\delta$ such that $\gamma = \delta \alpha_{n, i+1}$ and $\beta = \delta \alpha_{n, i}$. Consequently, $\beta - \gamma$ is contained in the two-sided ideal generated by $\alpha_{n, i} - \alpha_{n, i+1}$.
\end{proof}

The following result classifies all simple saturated $\C$-modules.

\begin{theorem} \label{simple saturated}
The submodule $\bigcap_{i \in [n]} K_i$ of $M(n)$ is a simple saturated $\C$-module. Furthermore, every simple saturated module is of this form up to isomorphism.
\end{theorem}

\begin{proof}
By the equivalence of $\C \module^{\rm{sa}}$ and $\C \fdmod$, it suffices to show that this intersection is saturated, and $\boldsymbol{\nu}$ sends it to the simple $\C$-module supported on $[n]$. For the first requirement, we note that there is an exact sequence
\[
0 \to \bigcap_{i \in [n]} K_i \to M(n) \to \bigoplus_{i \in [n]} k\C \alpha_{n-1, i}
\]
where the last map is the direct sum of those $f_i$. As each $k\C \alpha_{n-1, i}$ is torsion free, so is the image of $\bigoplus_{i \in [n]}f_i$. Therefore, by (2) of Lemma \ref{saturated modules}, $\bigcap_{i \in [n]} K_i$ is a saturated module.

Now we show the second requirement. Applying $\boldsymbol{\nu}$ we get an injection $\boldsymbol{\nu} \bigcap_{i \in [n]} K_i \to \boldsymbol{\nu} M(n)$. Note that $\boldsymbol{\nu} M(n) = D (e_n k\C)$, where $D = \Hom_k(-, k)$, is an indecomposable finite dimensional injective module.  By the structure of $\C$, $\boldsymbol{\nu} M(n)$ has a simple top on $[0]$ and a simple socle on $[n]$. Therefore, the second requirement is satisfied if the following conditions hold:
\begin{align}
\dim_k (\boldsymbol{\nu} \bigcap_{i \in [n]} K_i)_n & = \dim_k \Hom_{k\C} (\bigcap_{i \in [n]} K_i, M(n)) =1,\\
\dim_k (\boldsymbol{\nu} \bigcap_{i \in [n]} K_i)_{n-1} & = \dim_k \Hom_{k\C} (\bigcap_{i \in [n]} K_i, M(n-1)) = 0,
\end{align}
which are established in the next lemma.
\end{proof}

\begin{lemma}
Notations as above. Then (4.1) and (4.2) hold.
\end{lemma}

\begin{proof}
It is clear that $\dim_k \Hom_{k\C} (\bigcap_{i \in [n]} K_i, M(n)) \geqslant 1$. On the other hand, the short exact sequence
\[
0 \to \bigcap_{i \in [n]} K_i \to M(n) \to C \to 0
\]
where $C$ is the quotient module gives us another short exact sequence
\[
0 \to \Hom_{k\C}(\bigcap_{i \in [n]} K_i, M(n)) \to \Hom_{k\C}(M(n), M(n)) \to \Hom_{k\C}(C, M(n)) \to 0.
\]
Since the middle term has dimension 1, the dimension of the first term is at most one. This proves (4.1).

Similarly, to prove (4.2), it suffices to show that
\[
\Hom_{k\C} (C, M(n-1)) \cong \Hom_{k\C}(M(n), M(n-1)) \cong k^{\oplus n}.
\]
But this is also clear. Indeed, one can compose the inclusion map $C \to \bigoplus_{i \in [n]} k\C \alpha_{n-1, i}$ with the $n$ projection maps $\bigoplus_{i \in [n]} k\C \alpha_{n-1, i} \to k\C \alpha_{n-1, i}$ to get $n$ homomorphisms, and they are distinct because the $n$ distinct maps $f_i: M(n) \to k\C \alpha_{n-1, i}$ factor through them. Clearly, these homomorphisms span an $n$-dimensional space.
\end{proof}

We describe explicitly the simple saturated $\C$-modules for $n \in \{0, \, 1, \, 2\}$.

\begin{example} \normalfont
The simple saturated module $L^0$ corresponded to $n = 0$ is precisely $M(0)$. For $n = 1$, the simple saturated module $L^1$ is generated by $\alpha_{1, 1} - \alpha_{1, 2}$. More explicitly, for $i \geqslant 2$, its values on $[i]$ is spanned by elements of the form $\alpha - \beta$, where $\alpha, \beta \in \C([1], [i])$. For $n=2$, $L^2 = \langle \alpha_{2, 1} - \alpha_{2, 2} \rangle \cap \langle \alpha_{2, 2} - \alpha_{2, 3} \rangle$. By an explicit computation, we find that $L^2$ is generated by the element $(2,4) - (2, 3) + (1,3) - (1, 4) \in k\C([2], [4])$, where $(i_1, i_2)$ represents the function $[2] \to [4]$ sending 1 and 2 to $i_1$ and $i_2$ respectively.
\end{example}

\end{document}